\documentclass[preprint,11pt]{article}

\usepackage{lineno,hyperref}

\usepackage{graphicx, amssymb, amsmath}
\usepackage{indentfirst}
\usepackage{mathrsfs}
\usepackage{amsfonts}
\usepackage{color}
\usepackage{upgreek}
\usepackage{subfigure}
\usepackage[misc]{ifsym}
\usepackage{hyperref}
\usepackage{float}
\usepackage{xcolor}

\usepackage{epstopdf}
\usepackage{booktabs}
\usepackage{multirow}
\usepackage{makecell}
\usepackage{array}

\usepackage{geometry}
\usepackage{algorithmic}   
\usepackage[ruled,linesnumbered]{algorithm2e}

\usepackage[noblocks]{authblk} 

\newtheorem{theorem}{Theorem}[section]
\newtheorem{lemma}{Lemma}[section]

\newtheorem{corollary}{Corollary}[section]

\newtheorem{example}{\emph{Example}}[section]

\newcommand{\be}{\begin{equation}}
\newcommand{\ee}{\end{equation}}
\newcommand{\ba}{\begin{array}}
\newcommand{\ea}{\end{array}}
\newcommand{\beas}{\begin{eqnarray*}}
\newcommand{\eeas}{\end{eqnarray*}}
\newcommand{\bea}{\begin{eqnarray}}
\newcommand{\eea}{\end{eqnarray}}

\begin{document}
\title{ \bf Error Analysis of Virtual Element Method for the Poisson-Boltzmann Equation }

\author{Linghan Huang$^1$ ~Shi Shu$^2$~ Ying Yang$^{3,*}$
}

%

\footnotetext[1]{ School of Mathematics and Computational Science, Guilin University of Electronic Technology, Guangxi Colleges and Universities Key Laboratory of Data Analysis and Computation, Guangxi Applied Mathematics Center (GUET), Guilin, 541004, Guangxi, P. R. China. E-mail: 3466704709@qq.com}


\footnotetext[2]{ School of Mathematics and Computational Science, Xiangtan University, Hunan Key Laboratory for Computation and Simulation in Science and Engineering, Key Laboratory of Intelligent Computing and Information Processing of Ministry of Education, Xiangtan, 411105, Hunan, P. R. China. E-mail: shushi@xtu.edu.cn}

\footnotetext[3]
{$^*$ School of Mathematics and Computational Science, Guilin University of Electronic Technology, Guilin, Guangxi Colleges and Universities Key Laboratory of Data Analysis and Computation, Guangxi Applied Mathematics Center (GUET), 541004, Guangxi, P. R. China. E-mail: yangying@lsec.cc.ac.cn}





\date{}     
\maketitle

\noindent
\textbf{Abstract}:
 The Poisson-Boltzmann equation is a nonlinear elliptic equation with Dirac distribution sources, which has been widely applied to the prediction of electrostatics 
 potential of biological biomolecular systems in solution. In this paper, we discuss and analysis the virtual element method for the Poisson-Boltzmann equation on general polyhedral meshes. Under the low regularity of the solution of the whole domain, nearly optimal error estimates in both $L^2$-norm and $H^1$-norm for the virtual element approximation are obtained. The numerical experiment on different polyhedral meshes shows the efficiency of the virtual element method and verifies the proposed theoretical prediction.


\noindent

\section{Introduction}
\label{sec1}

The Poisson-Boltzmann equation (PBE) is a common tool in the study of biomolecular electrostatics, which provides an average field description of the electrostatic potential of biomolecular systems immersed in aqueous solutions \cite{Fogolari2002The}. In general, it is difficult to find the analytical solution of PBE and also a challenging task to solve it numerically, since it needs to deal with problems such as the strong nonlinearity, highly irregular interface, discontinuous coefficients and singular charges. Some numerical methods such as finite element (FE) method, finite difference (FD) method and boundary element (BE) method have been applied to solve PBE numerically. For example, Lu and McCammon \cite{Lu2007Improved} proposed an improved BE method to solve the PBE, which demonstrates considerable improvement in speed compared with the constant element and linear element methods. Chen, Holst and Xu \cite{Chen2007The} analyzed the FE approximation of the PBE and presented the first rigorous convergence result of the numerical discretization technique for the nonlinear PBE with delta distribution sources. Mirzadeh et. al. \cite{Mirzadeh2013Adaptive} presented an adaptive FD solver for the nonlinear PBE that uses non-graded, adaptive octree grids which drastically decrease memory usage and runtime without sacrificing accuracy compared to uniform grids. Kwon and Kwak \cite{Kwon2019Discontinuous} aimed at the PBE interface conditions, utilized the discontinuous bubble immersed FE method to obtain a discrete solution, and verified by examples that the optimal convergence rate was reached.

The commonly used discretization methods for PBE are based on triangular/quadrilateral or tetrahedral/hexahedral meshes. It would be more easier or flexible if polygonal/polyhedral mesh is applied in dealing with problems with the extremely irregular interface like PBE, but the implement of the traditional discretization method such as FD or FE method on such meshes would be difficult. For example, standard FE method can be applied in general polygons/polyhedras, but there are high requirements for the construction of shape functions. Recently, a generalized FE method called the virtual element method (VEM) is proposed by Veiga et. al. \cite{L2013Basic}, which can be used on arbitrary polygonal or polyhedral meshes. The novelty of the method lies in its ability to avoid constructing an explicit expression for the basis function. Instead, it only requires the appropriate degree of freedom to convert the discrete formulation into matrix form. This trait enables VEM to easily extend to higher order approximations and maintain robustness under general mesh types, even when dealing with complex geometries.
The VEM has been widely applied to solving many equations, such as second-order elliptic problems \cite{Antonietti2021Review,L2016Virtual,Cangiani2020Virtual,Cangiani2017Conforming}, stokes problems \cite{Kwak2022Formal,Liu2017Nonconforming,Manzini2021Virtual}, elastic equations \cite{Da2015Virtual,Nguyen2018Virtual,Van2022Mesh}, Maxwell equations \cite{Cao2021Virtual,Da2022Virtual}, plate bending problems \cite{Brezzi2013Virtual,Zhao2016Nonconforming} and Poisson-Nernst-Planck equations \cite{Liu2021Virtual}, etc. 



In this paper, we consider the VEM to solve the PBE. First, considering the singular term (Dirac distribution source) in the PBE, a decomposition technique is used to remove the singularities and leads to a regularized PBE (cf. \cite{Chen2007The}).
Then, we present the VEM discretization scheme for the homogeneous RPBE and analyze the numerical error for the VEM solution on general polyhedral meshes.
The main difficulties in the analysis for the RPBE include
the low global regularity of the solution and exponential rapid nonlinearities.
In order to present the error estimate of the solution with lower regularity, an interpolation error estimate is shown on arbitrary polyhedral meshes (see Lemma \ref{globalpi}), which is a generalization of the result on triangular meshes in \cite{Chen1998Finite}.
To deal with the exponential rapid nonlinear term, we follow the frame of the analysis in \cite{Cangiani2020Virtual}, but some properties of the nonlinear operator need to be studied in more details (see Lemma \ref{operatorB}), due to the difference of nonlinear forms compared with \cite{Cangiani2020Virtual}.
Finally, we show that both the $H^1$-norm and $L^2$-norm can reach the nearly optimal order with order $O(h|logh|^{1/2})$ and $O(h^2|logh|)$, respectively, when the interface is of $C^2$. The numerical example on three different polyhedral meshes verifies the theoretical convergence results and shows the efficiency of the virtual element method.

The outline of this paper is as follows. In Section 2, some preliminaries including the notations and some lemmas are presented. In Section 3, the detailed error estimates for the VEM solution are given. In Section 4, the numerical results are given to show the efficiency of the VEM and verify the proposed theoretical prediction.

\section{ Preliminaries }\label{sec2}
In this section, we shall first introduce some basic notations. For a open bounded domain $\mathcal{D} \subset \mathcal{R}^3$ with boundary $\partial \mathcal{D}$, we shall adopt the standard notations for Sobolev spaces $W^{s,p}(\mathcal{D})$ and their associated norms and seminorms. We denote $W^{s,p}_0=H^{s}_0(\mathcal{D})\bigcap W^{s,p}(\mathcal{D})$, for $p=2$, $H^s(\mathcal{D})=W^{s,2}(\mathcal{D})$, $H^1_0(\mathcal{D})=\{v\in H^1(\mathcal{D})|~v_{|_{\partial \mathcal{D}}}=0\}$, where $v_{|_{\partial \mathcal{D}}}=0$ is in the sense of trace, $\|\cdot\|_{s,p}= \|\cdot\|_{W^{s,p}(\mathcal{D})}$ (e.g. $\|\cdot\|_{0,\infty}=\|\cdot\|_{W^{0,\infty}(\mathcal{D})}$), and $(\cdot,\cdot)_{0,\mathcal{D}}$ is the standard $L^2$-inner product. For simplicity, $\|\cdot\|_{1}= \|\cdot\|_{W^{1,2}(\mathcal{D})}$ and $\|\cdot\|_{0}= \|\cdot\|_{L^2(\mathcal{D})}$.

 Concerning geometric objects (and related items), we shall use the following notations. For a geometric object $\mathcal{K}$ of dimension $d$, $1 \leq d \leq 3$
 (as an edge, or a face, or a polyhedron), we will denote by $\boldsymbol{x}_{\mathcal{K}}$, $h_{\mathcal{K}}$ and $|\mathcal{K}|$ the centroid, the diameter and the measure of $\mathcal{K}$, respectively. Moreover, for a non-negative integer $k$, let $\mathcal{P}_k(\mathcal{K})$ denote the space of polynomials of degree $\leq k$ on $\mathcal{K}$ and $\mathcal{M}_k(\mathcal{K})$ denote the set of scaled monomials
\[
    \mathcal{M}_k(\mathcal{K}):=\{m_{\boldsymbol{\alpha}} | m_{\boldsymbol{\alpha}} = \left( \frac{\boldsymbol{x}-\boldsymbol{x}_\mathcal{K}}{h_{\mathcal{K}}}\right)^{\boldsymbol{\alpha}},~0\leq |\boldsymbol{\alpha}|\leq k \},
\]
where, for a multi-index $\boldsymbol{\alpha} = (\alpha_1,...,\alpha_d)$, we denote, as usual, $|\boldsymbol{\alpha}|:=\alpha_1+\cdots+\alpha_d$ and $\boldsymbol{x}^{\boldsymbol{\alpha}}:=x_1^{\alpha_1}\cdots x_d^{\alpha_d}$. In addition, let
$\mathcal{P}_{-1}(\mathcal{K}):=\{0\}$, $\mathcal{M}_{-1}(\mathcal{K}):=\emptyset$.

Throughout this paper C denotes a positive constant independent of $h$, but may have different values at different places.

\subsection{Model problem and variational formulation}

 In this subsection, we shall introduce the nonlinear Poisson-Boltzmann equation (PBE), the boundary condition and the corresponding weak formulation. Let $\Omega \subset \mathcal{R}^3$ be a bounded convex polyhedron with boundary $\partial \Omega$ of $C^2$. Denote by $\Omega_m \subset \Omega$ and $\Omega_s = \Omega \setminus \Omega_m$ the molecule regin and the solvent region, respectively. We use $\tilde{u}$ to be the dimensionless potential and consider the following nonlinear Poisson-Boltzmann equation
\begin{equation}\label{nlpbe}
	-\nabla\cdot(\varepsilon\nabla\tilde{u})+\bar{\kappa}^2 \sinh(\tilde{u})=\sum_{i=1}^{N_{m}}\tilde{q_{i}}\delta_{i},\quad\text{in}\;\Omega,
\end{equation}
with Dirichlet boundary condition
\begin{equation}\label{nlb}
	\tilde{u}=0,\quad\text{on}\;\partial\Omega,
\end{equation}
and jump conditions on the interface $\Gamma := \partial \Omega_m = \Omega_s \cap \Omega_m$
\begin{equation}
    [\tilde{u}]=0,~[\varepsilon \frac{\partial \tilde{u}}{\partial n_{\Gamma}}]=0,
\end{equation}
where $[v] = \mathop{lim}\limits_{t \longrightarrow 0} v(x+tn_{\Gamma})-v(x-tn_{\Gamma})$, with $n_{\Gamma}$ being the unit outward normal direction of interface $\Gamma$. Assume $\Gamma$ is sufficiently smooth, say, of class $C^2$. The dielectric $\varepsilon$ and the modified Debye-H$\ddot{u}$ckel parameter $\bar{\kappa}$ are defined as follows:
	\begin{align}\label{parameters}
	\varepsilon=\begin{cases}
	\varepsilon_{m},~in~\Omega_{m},\\
	\varepsilon_{s},~in~\Omega_{s},\\
	\end{cases}~&
	\bar{\kappa}=\begin{cases}
	0,~in~\Omega_{m},\\
	\sqrt{\varepsilon_{s}}\kappa >0,~in~\Omega_{s},\\
	\end{cases}
	\end{align}
where $\kappa$ denotes the Debye-H$\ddot{u}$ckel parameter. The function $\delta_i := \delta(x-x_i)$ is a Dirac distribution at point $x_i$ and $\tilde{q}_i := \frac{4\pi e_c^2}{k_B T}z_i$, where $k_B > 0$ is the Boltzmann constant, $T$ is the temperature, $e_c$ is the unit of charge, and $z_i$ is the amount of charge.

To deal with the singularities of the $\delta$ distributions, we decompose $\tilde{u}$ into the following form (cf. \cite{Chen2007The}),
\begin{align}\label{G}
    \tilde{u}=\hat{u} + G,~~\text{with}~G = \sum_{i=1}^{N_m}G_i,
\end{align}
where $\hat{u}$ is an unknown function in $H^1(\Omega)$ and $G_i = \frac{\tilde{q}_i}{\varepsilon_m}\frac{1}{|x-x_i|}$ is the fundamental solution of
\[
    -\varepsilon_m \Delta G_i = \tilde{q}_i \delta_i,~\text{in}~\mathcal{R}^3.
\]
So the singularity of $\delta$ distribution are transferred to a known function $G$.

Substituting (\ref{G}) into (\ref{nlpbe})-(\ref{nlb}), we obtain the so-called regularized Poisson-Boltzmann equation (RPBE):
\begin{equation}\label{rpbe}
\begin{aligned}
    -\nabla\cdot(\varepsilon\nabla \hat{u})+\bar{\kappa}^2\sinh(\hat{u}+G) &= \nabla \cdot((\varepsilon-\varepsilon_{m})\nabla G),\quad\text{in}\;\Omega,\\
    \hat{u} &= -G,~~\text{on}~\partial \Omega.
\end{aligned}
\end{equation}


For the convenience of analysis, we consider the following homogeneous RPBE:
\begin{equation}\label{rrpbe}
\begin{aligned}
    -\nabla\cdot(\varepsilon\nabla u)+\bar{\kappa}^2\sinh(u+G) &= \nabla \cdot((\varepsilon-\varepsilon_{m})\nabla G),\quad\text{in}\;\Omega,\\
    u &= 0,~~\text{on}~\partial \Omega.
\end{aligned}
\end{equation}
The variational formulation of (\ref{rrpbe}) reads: find $u\in H^1_0(\Omega)$
such that
\begin{equation}\label{npbew}
    a(u,v)+(B(u),v)=(f_G,v),~\forall v \in H^1_0(\Omega),
\end{equation}
where
\begin{equation}\label{npbed}
a(u,v)=(\varepsilon \nabla u, \nabla v),~
(B(u),v)=(\bar{\kappa}^2\sinh(u+G),v),~
(f_{G},v)=(\nabla \cdot ((\varepsilon-\varepsilon_{m}) \nabla G),v).
\end{equation}
It is easy to show $f_G \in L^{\infty}(\Omega)$ from the definitions of $G$ and $\varepsilon$. The bilinear form $a(\cdot,\cdot)$ satisfies the coercivity and continuity conditions as follows (cf. \cite{Chen2007The}): for $u,v \in H^1_0(\Omega)$, there exist constants $C_*,C^*>0$ such that
\begin{align}\label{acc}
    C_*||u||_1^2 \leq a(u,u),~\text{and}~a(u,v) \leq C^*||u||_1||v||_1.
\end{align}

In \cite{Chen2007The}, Chen presented the existence and uniqueness of the solution of (\ref{npbew}) and also showed the
 a priori $L^{\infty}$-estimates as follows:
\begin{lemma}\label{uboundary} (cf. \cite{Chen2007The})
   Let $u\in H_0^1(\Omega)$ be the solution of (\ref{npbew}). Then u is in $L^{\infty}(\Omega)$.
\end{lemma}

For convenience of analysis, since then we shall use $\Omega_1, \Omega_2$ to denote $\Omega_m$ and $\Omega_s$, respectively. Define the following auxiliary space
\[
    X=H^1(\Omega) \cap H^2(\Omega_1) \cap H^2(\Omega_2),
\]
equipped with the norm
\begin{align*}
 ||v||_X &=||v||_{1,\Omega}+||v||_{2,\Omega_1}+||v||_{2,\Omega_2},~\forall v\in X.
\end{align*}

{\color{black}By using Lemma \ref{uboundary}, we have the regularity for the solution of the interface problem (\ref{rrpbe}).}
\begin{lemma}\label{ux}
The problem (\ref{rrpbe}) has a unique solution $u \in H^1_0(\Omega)$ and $u$ satisfies
\[
     ||u||_{X} \leq C.
\]
\end{lemma}
\begin{proof}
    From \cite{Chen1998Finite}, we obtain that
    \begin{align}\label{ux1}
    ||u||_X \leq C||f_G - B(u)||_{0,\Omega} \leq C( ||f_G||_{0,\Omega}+||B(u)||_{0,\Omega}).
    \end{align}
    Then by the definition of $B(u)$ in (\ref{npbed}), Lemma \ref{uboundary} and $G \in C^{\infty}(\Omega_2)$ we get
    \begin{align}\label{ux2}
       ||B(u)||_{0,\Omega} \leq C||sinh(u+G)||_{0,\Omega_2} \leq C.
    \end{align}
    Combining the facts that $f_G \in L^{\infty}(\Omega)$ and (\ref{ux1}) - (\ref{ux2}), we can get the desired result. $\hfill\Box$
\end{proof}
\subsection{Virtual elements and Discretization}

In the present subsection, we shall introduce the virtual element space and discretization of (\ref{npbew}). First, we give the detailed description of the decomposition $\mathcal{T}_h$.

Let $\mathcal{T}_h$ be a decomposition of $\bar{\Omega}$ made of non-overlapping and not self-intersecting polyhedral elements such that the diameter of any $E \in \mathcal{T}_h$ is bounded by $h:=\text{max}\{h_E:=\text{diam}(E) | E \in \mathcal{T}_h \}$. The faces and edges of element $E$ are denoted by $f$ and $e$, respectively. Following \cite{B2013Equivalent}, we make the assumption for the mesh $\mathcal{T}_h$: there exists a constant $\gamma > 0$ such that
    \begin{itemize}
      \item for each face $f$ and each edge $e$ on each element $E \in \mathcal{T}_h$ satisify, $h_e \geq \gamma h_f \geq \gamma^2 h_E,$
    where $h_f,~h_e$ are the diameter of the face $f$ and the diameter of the edge $e$, respectively.
      \item every element $E$ of $\mathcal{T}_h$ is star-shaped with respect to a ball of radius $ \geq \gamma h_E$.
      \item every face $f$ is star-shaped with respect to a ball of radius $\geq \gamma h_f$.
    \end{itemize}

Next, we introduce the definition of "interface element" (see Fig. \ref{interfaceele} as an example) to classify $\mathcal{T}_h$.
First, we assume the interface $\Gamma$ is represented by the zero-level set of a function $\varphi(x)$, i.e.,(cf. \cite{Chen2017An})
\[
      \Gamma := \{ x\in \Omega |~\varphi(x)=0\}.
\]
Then the subregions on $\Omega$ can be denoted by
\[
      \Omega_1:= \{ x\in \Omega |~\varphi(x)<0\}~and~\Omega_2:= \{ x\in \Omega |~\varphi(x)>0\}.
\]
Therefore, when an element $E$ satisfies $\bar{E}\cap \Gamma \neq \emptyset$, it is called an "interface element". More detailed description is as follows
\begin{itemize}
  \item There exists at least two point $x_1,~x_2 \in \bar{E}$, $E \in \mathcal{T}_h$ satisfy $\varphi(x_1)\varphi(x_2) \leq 0$.
\end{itemize}
Then, the decomposition $\mathcal{T}_h$ can be divided into
\begin{align}\label{ge}
\mathcal{T}_h^*:=\{E \in \mathcal{T}_h |~E~\text{is the interface element}\}
\end{align}
and $\mathcal{T}_h\setminus\mathcal{T}_h^*$.


\begin{figure}[H]
 \centering
 {
  \begin{minipage}{10cm}
   \centering
   \includegraphics[scale=0.8]{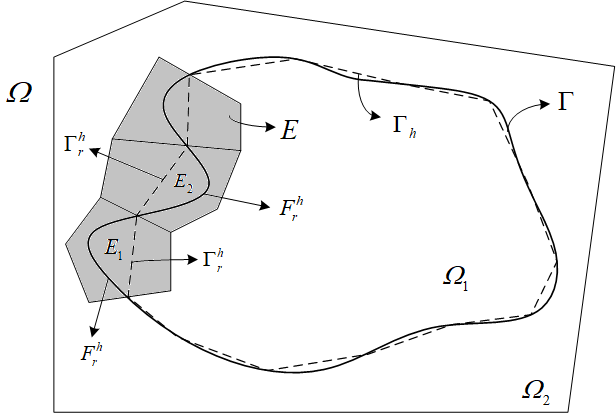}
  \end{minipage}
 }

    \caption{2-D schematic illustration of the interface elements. The grey region represents some of the interface elements. The interface $\Gamma$ is the black solid line that divides region $\Omega$ into $\Omega_1$ and $\Omega_2$, which is approximated by dotted line $\Gamma_h$.}\label{interfaceele}
\end{figure}
Since the interface $\Gamma$ is of class $C^2$, which is approximated by a union of polygons of $\mathcal{T}_h^*$, we assume that the discrete interface $\Gamma_h$ approximates the known interface $\Gamma$ to the second order, i.e.,(cf. Assumption 2, in \cite{Hiptmair2012Convergence})
\begin{align}\label{curve}
   dist(\Gamma,\Gamma_h)\leq Ch^2,
\end{align}
where the polygonal boundary $\Gamma_h$ with all vertices on $\Gamma$. Then, we make a further quantification of the region between $\Gamma$ and $\Gamma_h$. Let
\[
    \Gamma_h = \bigcup_{r=1}^{m_h}F_r^h ~\text{and}~ \Gamma = \bigcup_{r=1}^{m_h}\Gamma_r^h,
\]
where $F_r^h$ is the polygon in $\mathcal{T}_h^*$ and $\Gamma_r^h$ is the part of $\Gamma$ corresponding $F_r^h$ (see Fig.1). 
Next, we introduce the following properties of the interface element.
\begin{lemma}\label{intermeas}
For any interface element $E \in \mathcal{T}_{E}^*$, let $E_{i},~i=1,2$ denote the part between $\Gamma_r^h$ and $F_r^h$ that belongs to $\Omega_i$. Then, we can derive that
\begin{align}\label{fememeas}
either ~|E_{1}| \leq Ch_{E}^4,~or~|E_{2}| \leq Ch_{E}^4,
\end{align}
where $|E_{i}|$ is the Lebesgue measure of $E_{i}$.
\end{lemma}

\begin{proof} It is easy to get (\ref{fememeas}) by following the deduction in \cite{Feistauer1987On} for two dimensional domain $\Omega$.
 $\hfill\Box$
\end{proof}\\

%

Next, since the definition of virtual element space is based on the availability of certain local projection operators, we will introduce element-wise defined projectors. First for any face $f$, we need define a local space (cf. \cite{L2017High-order})
\[
    \tilde{V}_h^k(f) := \{v \in H^1(f) | v_{|_{\partial f}} \in C^0(\partial f),v_{|_e} \in \mathcal{P}_k(e),\forall e \in \partial f,\Delta v \in \mathcal{P}_{k}(f) \},
\]
and for $v_h \in \tilde{V}_h^k(f)$, choose the corresponding degrees of freedom as follows:
\begin{itemize}
  \item the values of $v_h$ at the vertices, for $k \geq 2$ the moments,
\end{itemize}
\begin{itemize}
  \item $|e|^{-1}(m,v_h)_{0,e}~\forall m \in \mathcal{M}_{k-2}(e)$,
  \item $|f|^{-1}(m,v_h)_{0,f}~\forall m \in \mathcal{M}_{k-2}(f)$.
\end{itemize}
Then we can define a projection operator
\[
    \Pi^{\nabla}_{k,f}:\tilde{V}_h(f) \rightarrow \mathcal{P}_k(f),
\]
associating any $v_h \in \tilde{V}_h(f)$ with the element in $\mathcal{P}_k(f)$ such that
\begin{align}\label{P1.1}
    (\nabla \Pi^{\nabla}_{k,f} v_h - \nabla v_h, \nabla p)_{0,f} = 0,~\forall p \in \mathcal{P}_k(f),
\end{align}
and
\begin{align}\label{P1.2}
    \int_{\partial f}(v_h - \Pi^{\nabla}_{k,f} v_h)\mathrm{d}s = 0~for~k=1,
\end{align}
or
\begin{align}\label{P1.3}
    \int_{f}(v_h - \Pi^{\nabla}_{k,f} v_h)\mathrm{d}\boldsymbol{x} = 0~for~k \geq 2.
\end{align}

Now, the local virtual element space on face $f$ is defined by
\[
    V_h(f):=\{v \in \tilde{V}_h^k(f)|(v-\Pi_{k,f}^{\nabla}v,p)_{0,f}=0,~\forall p \in \mathcal{M}_k(f)\setminus \mathcal{M}_{k-2}(f)\}.
\]
Similarly, we can give the definition of the local virtual element space on polyhedron $E$:
\[
    V_h(E):=\{ v \in \tilde{V}_h(E)| (v-\Pi_{k,E}^{\nabla}v,p)_{0,E}=0,~\forall p \in \mathcal{M}_k(E)\setminus \mathcal{M}_{k-2}(E)\},
\]
where
\[
\tilde{V}_h(E):=\{ v \in H^1(E) | v_{|_{\partial E}}\in C^0(\partial E), v_{|_{f}} \in V_h(f),\forall f \in \partial E, \Delta v \in \mathcal{P}_k(E) \},
\]
and $\Pi_{k,E}^{\nabla}: \tilde{V}_h(E) \rightarrow \mathcal{P}_k(E)$ is the projection operator defined on each element $E$, which satisfies conditions similar to (\ref{P1.1})-(\ref{P1.3}):
\begin{align}\label{P1.4}
    (\nabla \Pi^{\nabla}_{k,E} v_h - \nabla v_h, \nabla p)_{0,E} = 0,~\forall p \in \mathcal{P}_k(E),
\end{align}
\begin{align}\label{P1.5}
    \int_{\partial E}(v_h - \Pi^{\nabla}_{k,E} v_h)\mathrm{d}S = 0~for~k=1,
\end{align}
or
\begin{align}\label{P1.6}
    \int_{E}(v_h - \Pi^{\nabla}_{k,E} v_h)\mathrm{d}\boldsymbol{x} = 0~for~k \geq 2.
\end{align}
For simplicity, we use $\Pi_k^{\nabla}$ instead of $\Pi_{k,E}^{\nabla}$ in the following sections of the article.

For any $v_h \in \tilde{V}_h(E)$, choose the following degrees of freedom
\begin{itemize}
  \item the values of $v_h$ at the vertices of $E$,
\end{itemize}
and for $k \geq 2$ the moments
\begin{itemize}
  \item $|e|^{-1}(m,v_h)_{0,e}$, $\forall m \in \mathcal{M}_{k-2}(e)$ on each edge $e$ of $E$,
  \item $|f|^{-1}(m,v_h)_{0,f}$, $\forall m \in \mathcal{M}_{k-2}(f)$ on each face $f$ of $E$,
  \item $|E|^{-1}(m,v_h)_{0,E}$, $\forall m \in \mathcal{M}_{k-2}(E)$.
\end{itemize}

Finally, the virtual element space is defined by
\[
    V_h:=\{v\in H^1_0(\Omega)| v_{|_{E}}\in V_h(E),~\forall E \in \mathcal{T}_h\}.
\]

Before introducing the discrete form of (\ref{npbew}), the $L^2$ projection needs to be introduced. Let $\Pi_k^0$ be the $L^2$ projection onto $\mathcal{P}_k(E)$, for every $v_h \in V_h(E)$, defined by the orthogonality conditions
\begin{align}\label{Pik0}
    (v_h - \Pi_k^0 v_h,p)_{0,E}=0,~\forall p\in \mathcal{P}_k(E).
\end{align}
Then, we will give the discrete forms of (\ref{npbew}). First, we use the $a^E(\cdot,\cdot)$ to represent the restriction on $E$ in the corresponding bilinear form defined in (\ref{npbed}). Let $S^E(u_h,v_h)$ be a symmetric positive definite bilinear form on $V_h(E)\times V_h(E)$ and satisfy
\[
    C_*a^E(v_h,v_h) \leq S^E(v_h,v_h) \leq C^*a^E(v_h,v_h),~\forall v \in V_h(E),~\text{with}~\Pi_k^0v_h=0,
\]
for some positive constants $C_*,C^*$ independent of $E$ and $h_E$. Then the virtual element method of order $k \geq 1$ for (\ref{npbew}) reads: find $u_h \in {V}_h$ such that
\begin{align}\label{nlpbed}
    a_h(u_h,v_h) + (B_h(u_h),v_h) = (f_h,v_h),~\forall v_h \in {V}_h,
\end{align}
where the bilinear forms can be split as: for all $u_h,~v_h \in {V}_h$,
\begin{align}
    a_h(u_h,v_h) = \sum_E a^E_h(u_h,v_h),~
    (B_h(u_h),v_h) = \sum_E (B_h(u_h),v_h)_E,~
    (f_h,v_h) = \sum_E (f_h,v_h)_E. \notag
\end{align}
The definitions of the local forms on every element $E \in \mathcal{T}_h$ are as follows
\begin{align}
    a_h^E(u_h,v_h) &= (\varepsilon \Pi_{k-1}^0 \nabla u_h,\Pi_{k-1}^0 \nabla v_h)_{0,E} + S^E((I-\Pi_{k}^\nabla)u_h,(I-\Pi_{k}^\nabla)v_h),\notag\\
    (B_h(u_h),v_h)_{E} &= (B(\Pi_k^0u_h),\Pi_k^0v_h)_{0,E},\\
    (f_h,v_h)_E &= (f_G,\Pi_k^0v_h)_{0,E}.\notag
\end{align}

\subsection{Preliminary results}
In this subsection, we introduce some lemmas before presenting the error estimate for the PBE. We start by the following standard approximation result.
\begin{lemma}\label{pinature}(cf. \cite{Cangiani2017Conforming})
      There exists a positive constant $C$ such that, for any $E \in \mathcal{T}_h$ and any $v \in H^s(E)$, $s=0,1$, it holds:
      \begin{align*}
      ||v - \Pi_k^0 v||_{0,E} + h_E|v - \Pi_k^0 v|_{1,E} &\leq Ch_E^{s}|v|_{s,E},~1\leq s\leq k+1.
      \end{align*}
\end{lemma}

From \cite{L2013Basic}, the form $a^E_h(\cdot,\cdot)$ satisfies the following properties:
\begin{lemma}
Stability (cf. \cite{L2013Basic}): There exist positive constants $\alpha_*,\alpha^*$, independent of h and the element $E$ such that
      \begin{align}\label{as}
          \alpha_*a^E(v_h,v_h) \leq a_h^E(v_h,v_h) \leq \alpha^*a^E(v_h,v_h),~\forall v_h \in V_h(E).
      \end{align}
\end{lemma}

Furthermore, from (\ref{acc}) and (\ref{as}), it is easy to get the property of $a_h(\cdot,\cdot)$
\begin{align}\label{ahcc}
     C_*||v_h||_1^2 \leq a_h(v_h,v_h),~a_h(u_h,v_h) \leq C^*||u_h||_1||v_h||_1,~\forall u_h,v_h \in V_h,
\end{align}
where the constants $C_*,C^*>0$ are independent of $h$.

Since the coefficient $ \varepsilon >0 $ in (\ref{parameters}) is a bounded function, then we can get the following lemma.\begin{lemma}
k-Consistency: For all $p \in \mathcal{P}_k(E)$ and $v_h \in V_h(E)$, it holds
      \begin{align}\label{akc}
          a_h^E(p,v_h) = a^E(p,v_h).
      \end{align}
\end{lemma}

\begin{proof}
For any $E \in \mathcal{T}_h\setminus\mathcal{T}_h^*$, since the coefficient $\varepsilon$ is a constant, we have
\begin{align}\label{akc1}
     a^E(p,v_h)-a_h^E(p,v_h) &= (\varepsilon \nabla p,\nabla v_h)_{0,E}-(\varepsilon \Pi_{k-1}^0 \nabla p,\Pi_{k-1}^0 \nabla v_h)_{0,E} \notag\\
     &= (\varepsilon \nabla p,\nabla v_h)_{0,E} - (\varepsilon \nabla p,\nabla v_h)_{0,E} \notag\\
     &= 0.
\end{align}

Next we consider any $E$ in $\mathcal{T}_h^*$. Let $\mathcal{E}_1:=E\cap\Omega_1$ and $\mathcal{E}_2:=E\cap\Omega_2$. Similarly, according to the properties of $\Pi_k^0$, we get that: for any constant $C$ and $p \in \mathcal{P}_k(E)$, it holds
\begin{align}\label{akc2}
C\Pi_{k-1}^0\nabla p \equiv C\nabla p \equiv \Pi_{k-1}^0(C\nabla p).
\end{align}
Thus, combining (\ref{akc2}), we obtain that: for any $E \in \mathcal{T}_h^*$,
\begin{align}\label{akc3}
    a^E(p,v_h)-a_h^E(p,v_h)
    &= (\varepsilon \nabla p,\nabla v_h)_{0,E}-(\varepsilon \Pi_{k-1}^0 \nabla p,\Pi_{k-1}^0 \nabla v_h)_{0,E} \notag\\
    &= (\varepsilon \nabla p,\nabla v_h)_{0,E} - (\varepsilon_m \Pi_{k-1}^0\nabla p,\Pi_{k-1}^0\nabla v_h)_{0,\mathcal{E}_1} \notag\\
     &- (\varepsilon_s \Pi_{k-1}^0\nabla p,\Pi_{k-1}^0\nabla v_h)_{0,\mathcal{E}_2}\notag\\
     &= (\varepsilon \nabla p,\nabla v_h)_{0,E} - (\Pi_{k-1}^0(\varepsilon_m \nabla p),\Pi_{k-1}^0\nabla v_h)_{0,\mathcal{E}_1} \notag\\
     &- (\Pi_{k-1}^0(\varepsilon_s \nabla p),\Pi_{k-1}^0\nabla v_h)_{0,\mathcal{E}_2}\notag\\
     &= (\varepsilon \nabla p,\nabla v_h)_{0,E}-(\Pi_{k-1}^0(\varepsilon \nabla p),\Pi_{k-1}^0\nabla v_h)_{0,E}\notag\\
     &=0.
\end{align}
At last, combining (\ref{akc1}) and (\ref{akc3}), we finish the proof.    $\hfill\Box$
\end{proof}
\\
\noindent

Next, we give the following properties of the operator $B$ defined in (\ref{npbed}).
\begin{lemma}\label{operatorB}
1. (cf. \cite{Chen2007The}) For $u,v\in H^1(\Omega)$, the operator $B$ is monotone in the sense that
\begin{align}\label{pro1}
     (B(u)-B(v),u-v) \geq \bar{\kappa}^2||u-v||^2 \geq 0.
\end{align}

2. (cf. \cite{Chen2007The}) The operator $B$ is bounded in the sense that for $u,v \in L^{\infty}(\Omega)$, $w \in L^2(\Omega)$,
\begin{align}\label{pro2}
     (B(u)-B(v),w) \leq C||u-v||~||w||.
  \end{align}

3. Let $u$ and $u_h$ be the solutions to (\ref{npbew}) and (\ref{nlpbed}), respectively.
{\color{black}Define the following functions associated with operator B:
\[
   \bar{B}_u:=\int_0^1B_u(u-t(u-u_h))\mathrm{d}t,~\bar{B}_{uu}:=\int_0^1B_{uu}(u-t(u-u_h))\mathrm{d}t.
\]
}
If $u_h \in L^{\infty}(\Omega)$, then it holds
\begin{align}\label{pro3}
     ||\bar{B}_u||_{0,\infty} \leq C,~||\bar{B}_{uu}||_{0,\infty} \leq C,
\end{align}
and if $u,~u_h \in W^{1,\infty}(\Omega)$, it follows that
\begin{align}\label{pro4}
     ||\bar{B}_u||_{1,\infty} \leq C.
\end{align}
\end{lemma}

\begin{proof}
We only need to show (\ref{pro3}) and (\ref{pro4}), since (\ref{pro1}) and (\ref{pro2})  are given in \cite{Chen2007The}. According to the definition (\ref{npbed}) of operator $B$, it is easy to know
\begin{align}\label{Buu}
     B_u(u-t(u-u_h)) = \bar{\kappa}^2cosh(u-t(u-u_h)+G),\\ B_{uu}(u-t(u-u_h)=\bar{\kappa}^2sinh(u-t(u-u_h)+G).\notag
\end{align}
     From Lemma \ref{uboundary}, the assumption $u_h \in L^{\infty}(\Omega)$ and $G \in C^{\infty}(\Omega_2)$, we obtain
     \begin{align}\label{cosh}
     ||B_u(u-t(u-u_h))||_{0,\infty,\Omega} \leq C||cosh(u-t(u-u_h)+G)||_{0,\infty,\Omega_2}
     \leq C.
     \end{align}
It follows that
\begin{align}\label{barBu}
     |\bar{B}_u| = |\int_0^1B_u(u-t(u-u_h))\mathrm{d}t| \leq C \int_0^1\mathrm{d}t \leq C.
\end{align}

Similar to (\ref{cosh}), using the properties of $sinh$, we have
\begin{align*}
     ||B_{uu}(u-t(u-u_h))||_{0,\infty,\Omega} \leq C||sinh(u-t(u-u_h)+G)||_{0,\infty,\Omega_2}
     \leq C.
\end{align*}
Therefore,
\[
     |\bar{B}_{uu}| = |\int_0^1B_{uu}(u-t(u-u_h))\mathrm{d}t| \leq C \int_0^1\mathrm{d}t \leq C.
\]

Next, the derivative of $\bar{B}_u$ is expressed as follows,
\begin{align}\label{DBu1}
     D\bar{B}_u &= \int_0^1DB_u(u-t(u-u_h))\mathrm{d}t \notag\\
     &= \int_0^1\bar{\kappa}^2sinh(u-t(u-u_h)+G)(Du-t(Du-Du_h)+DG)\mathrm{d}t.
\end{align}
Since $u,~u_h \in W^{1,\infty}(\Omega)$ and $G\in L^{\infty}(\Omega_2)$, we obtain
\begin{align}\label{DBu2}
     ||\bar{\kappa}^2sinh(u-t(u-u_h)+G)||_{0,\infty,\Omega}\leq C||sinh(u-t(u-u_h)+G)||_{0,\infty,\Omega_2}\leq C,
\end{align}
\begin{align}\label{DBu3}
     ||Du-t(Du-Du_h)+DG||_{0,\infty,\Omega_2}\leq C.
\end{align}
Combining (\ref{DBu1})-(\ref{DBu3}), we get
\begin{align}\label{DBu4}
     |D\bar{B}_u| = |\int_0^1\bar{\kappa}^2sinh(u-t(u-u_h)+G)(Du-t(Du-Du_h)+DG)\mathrm{d}t|
     \leq C\int_0^1\mathrm{d}t \leq C.
\end{align}
Thus, by (\ref{barBu}) and (\ref{DBu4}), it is easy to show (\ref{pro4}).

This completes the proof. $\hfill\Box$
\end{proof}
\\

Now using the properties of operator $B$, we can derive the following bounded property with respect to $u_h$.
\begin{lemma}\label{uhH1}
     Let $u_h$ be the solution of (\ref{nlpbed}). If $u_h \in L^{\infty}(\Omega)$, then there holds
     \begin{align*}
         ||u_h||_{1,\Omega} \leq C.
     \end{align*}
\end{lemma}
\begin{proof}
     By (\ref{ahcc}) and taking $u = \Pi_k^0 u_h$, $v = 0$ in (\ref{pro2}), we get that
     \begin{align*}
         C_*||u_h||_{1,\Omega}^2 &\leq a_h(u_h,u_h) = (f_h,u_h)-(B_h(u_h),u_h) \\
         &\leq (||f_G||_{0,\Omega} + ||\Pi_k^0 u_h||_{0,\Omega})||u_h||_{0,\Omega} \\
         &\leq C (||f_G||_{0,\Omega} + ||u_h||_{0,\Omega})||u_h||_{1,\Omega}\\
         &\leq C||u_h||_{1,\Omega},
     \end{align*}
where the functions $f_G$ (see (\ref{npbed}) and $u_h$ are both in $L^{\infty}(\Omega)$. This completes the proof. $\hfill\Box$\\
\end{proof}

According to the classical Scott-Dupont theory (see \cite{Brenner2008The}), we have the following lemma.
\begin{lemma}\label{upilocal}(cf. \cite{Brenner2008The})
     For any $u \in H^s(E),s=1,2$, there exists a $\Pi_h u \in \mathcal{P}_k(E)$ such that
     \begin{align}\label{fem0}
         ||u-\Pi_{h}u||_{m,p,E} \leqslant Ch_E^{s-m}|u|_{s,p,E},~m=0,1,~1\leq p\leq \infty.
     \end{align}
\end{lemma}

Using Lemma \ref{upilocal}, we can derive an error estimate of the interpolant under a lower regularity of the solution, which is a tool in study of error estimate of the virtual element solution of the interface problem. To present the error estimate of the interpolant, we need to introduce the following Sobolev embedding inequality, a detailed proof of which in two-dimensions can be found in \cite{Chen1998Finite} and the same idea can be extended to 3D with no essential changes (see Lemma 3.1, in \cite{Ren1995On}).
\begin{lemma}(3D Sobolev Embedding Inequality)\label{sbi}
    For every $p > 2$ and all $\phi \in H^1(\Omega_i)$, $i=1,2$, it holds
    \begin{align*}
       ||\phi||_{0,p,\Omega_i} \leqslant Cp^{\frac{1}{2}}||\phi||_{1,\Omega_i}.
    \end{align*}
\end{lemma}

\setcounter{equation}{0}
\section{Error estimates}\label{errorestimates}

In this section, we shall present the $H^1$ and $L^2$ norm error estimates for the discrete solution $u_h$ under the assumption $\mathcal{T}_h$ is quasi-uniform.
{\color{black}Considering the low global regularity of the solution of problem (\ref{rrpbe}), 
	we give the following global error estimate of interpolation $\Pi_h u$ in the space $X$, which will be used in the error analysis of the virtual element approximation.
\begin{lemma}\label{globalpi}
   For any $u \in X$, there exists a $\Pi_h u \in \mathcal{P}_k(E)$ such that
   \begin{align}\label{pih0}
    (\sum_E||u-\Pi_h u||_{1,E}^2)^\frac{1}{2}\leq Ch|logh|^\frac{1}{2}||u||_X.
    \end{align}
\end{lemma}
\begin{proof}
First, for any $u \in X $, let $ u_{i} $ be the restriction of  $u$ on $\Omega_{i}$ for $i=1,2$. Since the interface $\Gamma$ is sufficiently smooth, we can extend $u_i \in H^2(\Omega_{i})$ onto the whole domain $\Omega$ and obtain the function $\tilde{u}_{i} \in H^2(\Omega)$ such that $\tilde{u}_{i}=u_{i}$ on $\Omega_{i}$ (cf. \cite{Chen1998Finite})
   \begin{align}\label{exten_u}
        ||\tilde{u}_{i}||_{2, \Omega} \leq C ||u||_{2, \Omega_{i}},~for~i=1,2.
   \end{align}

Next, we analyze the element $E$ in $\mathcal{T}_{h}^*$ and $\mathcal{T}_h \setminus \mathcal{T}_{h}^*$ respectively, where $\mathcal{T}_h^*$ is defined in (\ref{ge}). For any element $E$ in $\mathcal{T}_h \setminus \mathcal{T}_{h}^*$, we get from the standard finite element interpolation (cf. \cite{Brenner2008The})
\begin{align}\label{pih1}
	\|u-\Pi_{h}u\|_{m,E} \leq Ch_{E}^{2-m}||u||_{2,E},~m=0,1.
\end{align}
Then we consider any element $E$ in $\mathcal{T}_{h}^*$. Decompose the error as follows:
\begin{equation}\label{pih2}
     ||u-\Pi_{h}u||_{m,E} = ||u-\Pi_{h}u||_{m,E_i} + ||u-\Pi_{h}u||_{m,E \setminus E_i},~i=1,2,
\end{equation}
where $E_i$ is defined in Lemma \ref{intermeas}. By (\ref{fememeas}), without loss of generality, we may first assume that $|E_1|\leq Ch_E^4$. For any $p > 2$ and $m=0,1$, we get
\begin{align}\label{pih4}
	\|u-\Pi_{h}u\|_{m,E_1}^2
	&\leq \sum_{|\alpha|=0}^m(\int_{E_1}|D^\alpha(u-\Pi_{h}u)|^{2 \cdot \frac{p}{2}})^{\frac{2}{p}} (\int_{E_1} 1^{2 \cdot \frac{p}{p-2}})^{\frac{p-2}{p}} \notag\\
	&\leq Ch_{E}^{\frac{4(p-2)}{p}} \sum_{|\alpha|=0}^m||D^\alpha (u-\Pi_{h}u)||_{0,p,E_1}^2~~~(\text{by~(\ref{fememeas})}) \notag\\
	&\leq Ch_{E}^{\frac{4(p-2)}{p}} (\sum_{|\alpha|=0}^m||D^\alpha (u-\Pi_{h}u)||_{0,p,E_1}^{2 \cdot \frac{p}{2}})^{\frac{2}{p}} (\sum_{|\alpha|=0}^{m} 1^{\frac{p}{p-2}})^{\frac{p-2}{p}} \notag\\
	&\leq Ch_{E}^{\frac{4(p-2)}{p}} ||u-\Pi_{h}u||_{m,p,E_1}^2 \notag\\
    &\leq Ch_{E}^{\frac{4(p-2)}{p}+2(1-m)}||u||_{1,p,E}^2,
\end{align}
where we have used Lemma (\ref{upilocal}) in the last inequality. Furthermore, due to the quasi-uniformity of $\mathcal{T}_h$, we have the fact that
\[
\sum_{E \in \mathcal{T}_{E}^*}1\leq Ch^{-2}.
\]
Then using the discrete H\"older inequality, it yields
\begin{align}\label{pih6}
	\sum_{E\in \mathcal{T}_{E}^*}||u||_{1,p,E}^2 &\leq (\sum_{E\in \mathcal{T}_{E}^*} ||u||_{1,p,E}^{2 \cdot \frac{p}{2}})^{\frac{2}{p}} (\sum_{E\in \mathcal{T}_{E}^*} 1^{\frac{p}{p-2}})^{\frac{p-2}{p}}\notag \\
	&\leq Ch^{-\frac{2(p-2)}{p}} (\sum_{E\in \mathcal{T}_{E}^*} ||u||_{1,p,E}^p)^{\frac{2}{p}} \notag\\
	&\leq Ch^{-\frac{2(p-2)}{p}} ||u||_{1,p,\Omega}^2.
\end{align}
Combining (\ref{pih6}) and (\ref{pih4}), we obtain that
\begin{align}\label{pih7}
\sum_{E\in \mathcal{T}_{E}^*}||u-\Pi_h u||_{m,E_1}^2
&\leq C\sum_{E\in \mathcal{T}_{E}^*}h_{E}^{\frac{4(p-2)}{p}+2(1-m)}||u||_{1,p,E}^2 \notag\\
&\leq Ch^{4-\frac{4}{p}-2m} ||u||_{1,p,\Omega}^2.
\end{align}
On the other hand, using the extention $\tilde{u}_i$ of $u_i$ that for $m=0,1$, we get that
\begin{align}\label{pih3}
    ||u-\Pi_h u||_{m,E\setminus E_1}
    &= ||\tilde{u}_i-\Pi_h \tilde{u}_i||_{m,E\setminus E_1} \notag\\
    &\leq ||\tilde{u}_i-\Pi_h \tilde{u}_i||_{m,E} \notag\\
    &\leq Ch_{E}^{4-2m}||\tilde{u}_i||_{2,E}.
\end{align}
So for the region $\mathcal{T}_h^*$, from (\ref{pih2}), (\ref{pih7}) and (\ref{pih3}), we get
\begin{align}\label{pih8}
	\sum_{E\in \mathcal{T}_h^*}||u-\Pi_h u||_{m,E}^2
	&\leq Ch^{4-2m}\sum_{E\in \mathcal{T}_h^*}||\tilde{u}_i||_{2,E}^2 + Ch^{4-\frac{4}{p}-2m}||u||_{1,p,\Omega}^2 \notag\\
	&\leq Ch^{4-2m}||\tilde{u}_i||_{2,\Omega}^2 + Ch^{4-\frac{4}{p}-2m} ||u||_{1,p,\Omega}^2.
\end{align}

From Lemma (\ref{sbi}) and (\ref{exten_u}), it is easy to know that
\begin{align}\label{pih9}
 ||\tilde{u}_{i}||_{2,\Omega} \leq C||u||_{2,\Omega_{i}} \leq C||u||_{X},
\end{align}
and
\begin{align}\label{pih9.1}
 ||u||_{1,p,\Omega}^p &= ||\tilde{u}_1||_{1,p,\Omega_1}^p + ||\tilde{u}_2||_{1,p,\Omega_2}^p,\notag\\
&=\sum_{i=1}^2 (||\tilde{u}_i||_{0,p,\Omega_{i}}^p + ||\nabla \tilde{u}_i||_{0,p,\Omega_{i}}^p)\notag\\
&\leq Cp^{\frac{p}{2}}\sum_{i=1}^2(||\tilde{u}_i||_{1,\Omega_{i}}^p + ||\nabla \tilde{u}_i||_{1,\Omega_{i}}^p)\notag \\
&\leq Cp^{\frac{p}{2}}||u||_{X}^p.
\end{align}
Combining (\ref{pih8})-(\ref{pih9.1}), we obtain that
\begin{align}\label{pih10}
\sum_{E\in \mathcal{T}_h^*}||u-\Pi_h u||_{m,E}^2
&\leq Ch^{4-2m}||u||_{X}^2 + Ch^{4-\frac{4}{p}-2m}p||u||_{X}^2.
\end{align}
Then it follows from (\ref{pih1}) and (\ref{pih10}),
\begin{align*}
	\sum_{E}\|u-\Pi_h u\|_{m,E}^2
    &= \sum_{E \in \mathcal{T} \setminus \mathcal{T}_h^*}||u-\Pi_h u||_{m,E}^2 + \sum_{E \in \mathcal{T}_h^*}||u-\Pi_h u||_{m,E}^2\\
	&\leq Ch^{4-2m}\sum_{E \in \mathcal{T} \setminus \mathcal{T}_h^*}||u||_{2,E}^2 + Ch^{4-2m}||u||_{X}^2 + Ch^{4-\frac{4}{p}-2m}p||u||_{X}^2\\
	&\leq Ch^{4-2m}||u||_{X}^2 + Ch^{4-\frac{4}{p}-2m}p||u||_{X}^2.
\end{align*}
Thus, we conclude that for $m=0,1$ and any $p>2$,
\begin{align}\label{pih11}
    (\sum_{E}\|u-\Pi_h u\|_{m,E}^2)^\frac{1}{2} \leq Ch^{2-m}||u||_{X} + Ch^{2-m-\frac{2}{p}}p^{\frac{1}{2}}||u||_{X}.
\end{align}
At last, it is easy to show that for any fixed sufficiently small $h$, $h^{2-m-\frac{2}{p}}p^{\frac{1}{2}}$ achieves its minimum $2h^{2-m-\frac{1}{2|logh|}}|logh|^{\frac{1}{2}}$ when $p = 4|logh|$ in (\ref{pih11})
and $h^{-\frac{1}{2|logh|}}\equiv \sqrt{e}$ in $h\in (0,1)$, then we obtain for $m=0,1$ that
\[
    (\sum_{E}\|u-\Pi_h u\|_{m,E}^2)^\frac{1}{2} \leq Ch^{2-m}|logh|^{\frac{1}{2}}||u||_{X}.
\]
This completes the proof.      $\hfill\Box$\\

\end{proof}

{\color{black}Now we begin to prove the following preliminary $H^1$-norm error bound.}
\begin{theorem}\label{H1error}
Let $u$ and $u_h$ be solutions of (\ref{npbew}) and (\ref{nlpbed}), respectively. If $u_h$ is in $L^{\infty}(\Omega)$, the following error estimate holds:
\begin{align}\label{Th1}
	||u-u_h||_{1,\Omega}\leq C(h|logh|^{\frac{1}{2}}+||u-u_h||_{0,\Omega}).
\end{align}
\end{theorem}
\begin{proof}
Splitting the error into the following form:
\begin{align}\label{sp1}
    u-u_h=u-\Pi_h u+\Pi_h u-u_h.
\end{align}
Then set $e_h=u_h-\Pi_h u$, and it follows from (\ref{npbew}), (\ref{nlpbed}) and (\ref{ahcc})-(\ref{akc})
\begin{align}\label{Th11}
    C\sum_E||e_h||_{1,E}^2 &\leq \sum_Ea_h^E(e_h,e_h)\notag\\
	&= \sum_Ea_h^E(u_h,e_h)-\sum_Ea_h^E(\Pi_hu,e_h)\notag\\
	&= \sum_Ea_h^E(u_h,e_h)-\sum_Ea^E(u,e_h)+\sum_Ea^E(u,e_h)-\sum_Ea^E(\Pi_hu,e_h)\notag\\
	&= \sum_E[(f_h,e_h)_E-(f_G,e_h)_{0,E}]+\sum_E[(B(u),e_h)_{0,E}-(B_h(u_h),e_h)_E]\notag\\
    &+\sum_Ea^E(u-\Pi_hu,e_h)\notag\\
	&=: III_1+III_2+III_3.
\end{align}
Next, we estimate $III_1$-$III_3$, respectively. From Lemma \ref{pinature}, we deduce that
\begin{align}\label{Th12}
	III_1 = \sum_E(f_G,\Pi_k^0e_h-e_h)_{0,E} \leq \sum_E||f_G||_{0,E}||e_h-\Pi_k^0e_h||_{0,E}\leq Ch||f_G||_{0,\Omega}(\sum_E||e_h||_{1,E}^2)^{\frac{1}{2}}.
\end{align}
Then for the second term $III_2$,
\begin{align}\label{Th13}
	(B(u),e_h)_{0,E}-(B_h(u_h),e_h)_E
	&=(B(u),e_h)_{0,E}-(B(\Pi_k^0u_h),\Pi_k^0e_h)_{0,E} \notag\\
    &=(B(u)-B(\Pi_k^0u),e_h)_{0,E}+(B(\Pi_k^0u),e_h-\Pi_k^0e_h)_{0,E}\notag\\
    &+(B(\Pi_k^0u)-B(\Pi_k^0u_h),\Pi_k^0e_h)_{0,E} \notag\\
    &=:III_{2.1}+III_{2.2}+III_{2.3}.
\end{align}
Using the Cauchy-Schwarz inequality, (\ref{pro2}) and the propoerties of operator $\Pi_k^0$ to estimate (\ref{Th13}), we obtain that
\begin{align}\label{Th14}
  III_{2.1} &= (B(u)-B(\Pi_k^0u),e_h)_{0,E}\notag\\
  &\leq C||u-\Pi_k^0u||_{0,E}||e_h||_{1,E}\notag\\
  &\leq Ch_E||u||_{1,E}||e_h||_{1,E},
\end{align}
\begin{align}\label{Th15}
  III_{2.2} &= (B(\Pi_k^0u),e_h-\Pi_k^0e_h)_{0,E}\notag\\
  &\leq C||\Pi_k^0u||_{0,E}||e_h-\Pi_k^0e_h||_{0,E}\notag\\
  &\leq Ch_E||u||_{0,E}||e_h||_{1,E},
\end{align}
and
\begin{align}\label{Th16}
  III_{2.3} &= (B(\Pi_k^0u)-B(\Pi_k^0u_h),\Pi_k^0e_h)_{0,E}\notag\\
  &\leq C||\Pi_k^0u- \Pi_k^0u_h||_{0,E}||e_h||_{1,E}\notag\\
  &\leq C||u-u_h||_{0,E}||e_h||_{1,E}.
\end{align}
Thus, we derive from (\ref{Th13})-(\ref{Th16}) that
\begin{align}\label{Th17}
	III_2 &= \sum_E[(B(u),e_h)_{0,E}-(B_h(u_h),e_h)_E]\notag\\
	&\leq Ch\sum_E||u||_{1,E}||e_h||_{1,E} + C\sum_E||u-u_h||_{0,E}||e_h||_{1,E}\notag\\
	&\leq C(h||u||_{X}+||u-u_h||_{0,\Omega}) (\sum_E||e_h||_{1,E}^2)^{\frac{1}{2}}.
\end{align}

To estimate the last term $III_3$ in (\ref{Th11}), using Cauchy-Schwarz inequality and Lemma \ref{globalpi} yield
\begin{align}\label{Th18}
    III_3 &= \sum_E(\varepsilon \nabla(u-\Pi_h u),\nabla e_h)_{0,E} \notag\\
    &\leq C\sum_E||u-\Pi_h u||_{1,E}||e_h||_{1,E} \notag\\
    &\leq Ch|logh|^{\frac{1}{2}}||u||_X (\sum_E||e_h||_{1,E}^2)^{\frac{1}{2}}.
\end{align}
Then, from (\ref{Th11})-(\ref{Th12}), (\ref{Th17})-(\ref{Th18}) and Lemma \ref{ux}, it follows that
\begin{align}\label{Th19}
	\sum_E||e_h||_{1,E}^2
	&\leq C(h||f_G||_{0,\Omega}+||u-u_h||_{0,\Omega}+h|logh|^{\frac{1}{2}}||u||_X) (\sum_E||e_h||_{1,E}^2)^{\frac{1}{2}} \notag\\
	&\leq C(h|logh|^{\frac{1}{2}}+||u-u_h||_{0,\Omega}) (\sum_E||e_h||_{1,E}^2)^{\frac{1}{2}},
\end{align}
where $f_G \in L^{\infty}(\Omega)$ is used. Now the $H^1$-error estimate (\ref{Th1}) follows from (\ref{sp1}), (\ref{Th19}) and Lemma \ref{globalpi}. This completes the proof.  $\hfill\Box$\\
\end{proof}

Next, we shall prove the following $L^2$-norm error bound.
\begin{theorem}\label{L2error}
Let $u$ and $u_h$ be the solutions of (\ref{npbew}) and (\ref{nlpbed}), respectively. If $u_h$ is in $W^{1,\infty}(\Omega)$,
there holds
\begin{align}\label{Th2}
	||u-u_h||_{0,\Omega}\leq C(h^2|\log h|+||u-u_h||_{0,\Omega}^3).
\end{align}
\end{theorem}
\begin{proof}
In order to show the $L^2$-estimate in (\ref{Th2}), we first consider the auxiliary problem: find $\psi \in H^1_0(\Omega)$, such that
\begin{align}\label{au1}
	(\varepsilon \nabla \psi,\nabla v)+(B_u(u)\psi,v)=(u-u_h,v),~\forall v\in H^1_0(\Omega),
\end{align}
where $B_u(u):=\bar{\kappa}^2\cosh(u+G)$. In fact, since $u\in L^{\infty}(\Omega)$ and $G \in L^{\infty}(\Omega_2)$, we get
\[
0 \leq ||B_u(u)||_{0,\infty,\Omega} \leq C||cosh(u+G)||_{0,\infty,\Omega_2}\leq C.
\]
Then the existence and uniqueness of the solution of problem (\ref{au1}) hold (cf. \cite{Ciarlet1978The}).

Next, we present the regularity for the solution $\psi$ of (\ref{au1}). It is easy to know
\[
C||\psi||_{1,\Omega}^2 \leq (\varepsilon \nabla \psi,\nabla v)+(B_u(u)\psi,v) = (u-u_h,\psi).
\]
Thus,
\begin{align}\label{aux1}
||\psi||_{1,\Omega} \leq C||u-u_h||_{0,\Omega}.
\end{align}
Since $\partial \Omega$ is of $C^2$, the solution $\psi$ of (\ref{au1}) satsifies (cf. Theorem \cite{Babuska1970The})
\begin{align}\label{aux2}
    ||\psi||_{2,\Omega_i} \leq C(||\psi||_{1,\Omega}+||u-u_h||_{0,\Omega})\leq C||u-u_h||_{0,\Omega},~i=1,2.
\end{align}
Finally, combining (\ref{aux1}) with (\ref{aux2}), we get that
\begin{align}\label{psiX}
||\psi||_X = ||\psi||_{1,\Omega} + ||\psi||_{2,\Omega_1} + ||\psi||_{2,\Omega_2} \leq C||u-u_h||_{0,\Omega}.
\end{align}

Now we show the $L^2$-estimate (\ref{Th2}). Taking $v=u-u_h$ in (\ref{au1}) and using (\ref{npbew}) and (\ref{nlpbed}), we have
\begin{align}\label{Th21}
    ||u-u_h||_{0,\Omega}^2
	&=(\varepsilon \nabla(u-u_h),\nabla \psi)+(B_u(u)(u-u_h),\psi)\notag\\
	&=\sum_Ea^E(u-u_h,\psi-\Pi_h\psi)+\sum_Ea^E(u-u_h,\Pi_h\psi)+(B_u(u)(u-u_h),\psi)\notag\\
	&=\sum_Ea^E(u-u_h,\psi-\Pi_h\psi)+\sum_Ea^E(u,\Pi_h\psi)-\sum_Ea_h^E(u_h,\Pi_h\psi)\notag\\
	&+(B_u(u)(u-u_h),\psi)\notag\\
	&=\sum_Ea^E(u-u_h,\psi-\Pi_h\psi)+\sum_E(f_h-f_G,\Pi_h\psi)_{0,E}+\sum_E(B_h(u_h)-B(u),\Pi_h\psi)_{0,E}\notag\\
	&+(B_u(u)(u-u_h),\psi)\notag\\
	&=\sum_Ea^E(u-u_h,\psi-\Pi_h\psi)+\sum_E(B_h(u_h)-B(u_h),\Pi_h\psi)_{0,E}\notag\\
	&+\sum_E(B(u_h)-B(u),\Pi_h\psi-\psi)_{0,E}+\sum_E(B(u_h)-B(u)+B_u(u)(u-u_h),\psi)_{0,E}\notag\\
	&=:III_4+III_5+III_6+III_7,
\end{align}
where $(f_h-f_G,\Pi_h\psi)=0$ is used. By Lemmas \ref{globalpi} and (\ref{psiX}), we can bound the term $III_4$ by
\begin{align}\label{Th22}
	III_4 &\leq C\sum_E||u-u_h||_{1,E}||\psi-\Pi_h\psi||_{1,E} \notag\\
&\leq Ch|logh|^{\frac{1}{2}}||u-u_h||_{1,\Omega}||\psi||_X \notag\\
&\leq Ch|logh|^{\frac{1}{2}}||u-u_h||_{1,\Omega}||u-u_h||_{0,\Omega}.
\end{align}

For the term $III_5$, it can be changed into the following form:
\begin{align}\label{Th23}
	III_5
    &= \sum_E(B(\Pi_k^0u_h)-B(u_h),\Pi_h\psi)_{0,E}\notag\\
    &= \sum_E((\Pi_k^0u_h - u_h)\bar{B}_{u_h},\Pi_h\psi)_{0,E},
\end{align}
where {\color{black}
\[
    B(\Pi_k^0u_h)-B(u_h)=(\Pi_k^0u_h - u_h)\int_0^1B_{u_h}(\Pi_k^0u_h-t(\Pi_k^0u_h-u_h))\mathrm{d}t=(\Pi_k^0u_h - u_h)\bar{B}_{u_h},
\]}
\[
B_{u_h}(\Pi_k^0u_h-t(\Pi_k^0u_h-u_h))=\bar{\kappa}^2cosh(\Pi_k^0u_h-t(\Pi_k^0u_h-u_h)+G).
\]
Similar to (\ref{pro4}), it is easy to show that $\bar{B}_{u_h} \in W^{1,\infty}(\Omega)$ under the condition $u_h \in W^{1,\infty}(\Omega)$. Furthermore, define a linear interpolation of $\bar{B}_{u_h}$:
\[
\bar{B}_{avg}:=\frac{1}{|E|}\int_E \bar{B}_{u_h},
\]
satisfying (cf. \cite{Verfurth1994Posteriori})
\[
    ||\bar{B}_{u_h}-\bar{B}_{avg}||_{0,\infty,E} \leq Ch_E||\bar{B}_{u_h}||_{1,\infty,\tilde{w}_E},
\]
where $\tilde{w}_E \in \Omega$ denote the union of all elements having a nonempty intersection with $E$. Then, from the property (\ref{Pik0}) of $\Pi_k^0$ and (\ref{pro3}), (\ref{Th23}) becomes
\begin{align}\label{Th23.1}
    III_5
    &= \sum_E((\Pi_k^0u_h - u_h)(\bar{B}_{u_h} - \bar{B}_{avg}),\Pi_h\psi)_{0,E} \notag\\
	&\leq C\sum_E||u_h - \Pi_k^0u_h||_{0,E}||\bar{B}_{u_h} - \bar{B}_{avg}||_{0,\infty,E}||\Pi_h\psi||_{0,E} \notag\\
	&\leq C\sum_E(||u-u_h||_{0,E} + h_E||u||_{1,E})h_E||\bar{B}_{u_h}||_{1,\infty,\tilde{w}_E}||\psi||_{0,E}\notag\\
	&\leq Ch(||u-u_h||_{1,\Omega}+h||u||_{X})||u-u_h||_{0,\Omega}.
\end{align}

For the term $III_6$, from (\ref{pro2}) and Lemma \ref{globalpi}, we get
\begin{align}\label{Th24}
	III_6 &\leq \sum_E||B(u)-B(u_h)||_{0,E} ||\Pi_h\psi-\psi||_{0,E}\notag\\
	&\leq C\sum_E||u-u_h||_{0,E}||\Pi_h\psi-\psi||_{0,E}\notag\\
	&\leq Ch|logh|^{\frac{1}{2}}||u-u_h||_{0,\Omega} ||\psi||_X\notag\\
	&\leq Ch|logh|^{\frac{1}{2}}||u-u_h||_{0,\Omega}^2.
\end{align}
For the last term $III_7$, using (\ref{pro3}),
we have
\begin{align*}
	III_7 &= (B(u_h)-B(u)+B_u(u)(u-u_h),\psi) \notag\\
	&= |(\bar{B}_{uu}(u-u_h)^2,\psi)| \notag\\
    &\leq C||(u-u_h)^2||_{0,\Omega}||\psi||_{0,\Omega}\notag\\
    &\leq C||u-u_h||_{0,3,\Omega}||u-u_h||_{0,6,\Omega}||\psi||_X,
\end{align*}
where
\[
    B(u) - B(u_h) - B_u(u)(u-u_h) = (u-u_h)^2 \int_0^1B_{uu}(u-t(u-u_h))\mathrm{d}t = (u-u_h)^2\bar{B}_{uu},
\]
and we have used the following H\"older inequality in the last inequality,
\[
   ||vw|| \leq ||v||_{0,3}||w||_{0,6}.
\]
Then, in view of the Gagliardo-Nirenberg-Sobolev inequality,
\[
||v||_{0,3} \leq C||v||_0^{\frac{1}{2}}||v||_1^{\frac{1}{2}},
\]
the Sobolev Imbedding Theorem
\[
   ||v||_{0,6} \leq C||v||_{1},
\]
and the regularity of the auxiliary problem (\ref{au1}), we have
\begin{align}\label{Th25}
    III_7
	&\leq C||u-u_h||_{0,\Omega}^{\frac{1}{2}} ||u-u_h||_{1,\Omega}^{\frac{3}{2}}||\psi||_X \notag\\
	&\leq C||u-u_h||_{0,\Omega}^{\frac{1}{2}} ||u-u_h||_{1,\Omega}^{\frac{3}{2}}||u-u_h||_{0,\Omega}.
\end{align}
Inserting (\ref{Th22}) and (\ref{Th23.1})-(\ref{Th25}) into (\ref{Th21}), it follows
\begin{align}\label{Th26}
	||u-u_h||_{0,\Omega} &\leq C(h|logh|^{\frac{1}{2}}||u-u_h||_{1,\Omega} + h(||u-u_h||_{1,\Omega}+h||u||_X)
	+ ||u-u_h||_{0,\Omega}^{\frac{1}{2}} ||u-u_h||_{1,\Omega}^{\frac{3}{2}})\notag\\
	&\leq C(h|logh|^{\frac{1}{2}}||u-u_h||_{1,\Omega} + ||u-u_h||_{0,\Omega}^{\frac{1}{2}} ||u-u_h||_{1,\Omega}^{\frac{3}{2}}+h^2||u||_{X}) \notag\\
    &=: III_8 + III_9 + Ch^2||u||_{X}.
\end{align}
Combining the result of Theorem \ref{H1error}, we get the following results. For $III_8$ in (\ref{Th26}), we get
\begin{align}
    III_8 \leq C(h^2|logh| + h|logh|^\frac{1}{2}||u-u_h||_{0,\Omega}).
\end{align}
As for term $III_9$, using the Cauchy's inequality with $\epsilon > 0$ (cf. \cite{Evans2010Partial}), we have
\begin{align}\label{Th27}
    III_9 &\leq C(\epsilon||u-u_h||_{0,\Omega}+\frac{1}{\epsilon}||u-u_h||_{1,\Omega}^3) \notag\\
    &\leq C(\epsilon||u-u_h||_{0,\Omega}+\frac{1}{\epsilon}(h|logh|^{\frac{1}{2}}+||u-u_h||_{0,\Omega})^3)\notag\\
    &=: C(\epsilon||u-u_h||_{0,\Omega}+\frac{1}{\epsilon} III_{10}).
\end{align}
For $III_{10}$ in (\ref{Th27}), it's easy to know that
\begin{align}\label{Th28}
    III_{10} &= (h|logh|^{\frac{1}{2}}+||u-u_h||_{0,\Omega})(h^2|logh|+||u-u_h||_{0,\Omega}^2+2h|logh|^{\frac{1}{2}}||u-u_h||_{0,\Omega})\notag\\
    &\leq 2(h|logh|^{\frac{1}{2}}+||u-u_h||_{0,\Omega})(h^2|logh|+||u-u_h||_{0,\Omega}^2) \notag\\
    &\leq
    2h^3|logh|^{\frac{3}{2}}+3||u-u_h||_{0,\Omega}^3+3h^2|logh|~||u-u_h||_{0,\Omega} \notag\\
    &\leq Ch^2|logh| + C||u-u_h||_{0,\Omega}^3 + Ch|logh|^{\frac{1}{2}}||u-u_h||_{0,\Omega},
\end{align}
where we have used for $h \in (0,1)$, $h|logh|^{\frac{1}{2}}>h^2|logh|>h^3|logh|^{\frac{3}{2}}$.

At last, combining (\ref{Th26}) - (\ref{Th28}) and Lemma \ref{ux}, we get
\begin{align*}
    ||u-u_h||_{0,\Omega} &\leq Ch^2|logh|+C(\epsilon+h|logh|^{\frac{1}{2}})||u-u_h||_{0,\Omega}+C||u-u_h||_{0,\Omega}^3.
\end{align*}
Choose $h$ and $\epsilon$ sufficiently small such that the term $C(\epsilon+h|logh|^{\frac{1}{2}})\ll 1$, then we get the desired estimate.
This completes the proof. $\hfill\Box$
\end{proof}
\\

In order to show the error estimates in $H^1$ and $L^2$-norms, it remains to demonstrate that $u_h$ converges to $u$.
\begin{theorem}\label{uuhconvergence}
Let $u$ and $u_h$ be solutions of (\ref{npbew}) and (\ref{nlpbed}), respectively. If $u_h$ is in $L^{\infty}(\Omega)$, the $u_h$ converges to $u$ in $H^1_0(\Omega)$.
\end{theorem}
\begin{proof}
We follow the arguments in \cite{Cangiani2020Virtual} to present the convergence of $u_h$. From Lemma \ref{uhH1}, since $||u_h||_{1,\Omega}$ is bounded, we can choose a subsequence $u_{h_k}$ such that for some $w \in H^1_0(\Omega)$, $u_{h_k} \rightarrow w$, weakly in $H^1_0(\Omega)$, as $h_{k} \rightarrow 0$ and, thus, strongly in $L^2(\Omega)$. Also let an abritrary $v \in C^{\infty}_0(\Omega)$ and $v_{h_k}$ a sequence in $V_{h_k}$ such that
\begin{align}\label{Th31}
	||v-v_{h_k}||_{1,\Omega} \rightarrow 0,~~~h_{k} \rightarrow 0.
\end{align}
Next, we shall prove that $w$ is the weak solution of problem (\ref{npbew}). First, there holds
\begin{align*}
	|a(w,v)+(B(w),v)&-(f_G,v)|
	\leq |a(w,v-v_{h_k})|+|a(w,v_{h_k})-a_h(u_{h_k},v_{h_k})|\notag\\ &+|(B(w),v)-(B_h(u_{h_k}),v_{h_k})|+|\sum_E(f_G,\Pi_k^0v_{h_k}-v_{h_k})_{0,E}|+|(f_G,v_{h_k}-v)| \notag\\
	&\leq C||\nabla w||_{0,\Omega}||v-v_{h_k}||_{1,\Omega} + Ch_k||f_G||_{0,\Omega}||v_{h_k}||_{1,\Omega}+C||f_G||_{0,\Omega}||v-v_{h_k}||_{1,\Omega}\notag\\
	&+|a(w,v_{h_k})-a_h(u_{h_k},v_{h_k})| + |(B(w),v)-(B_h(u_{h_k}),v_{h_k})|.
\end{align*}
Then from (\ref{Th31}), $w$ is the weak solution of (\ref{npbew}) if
\begin{align}\label{Th32}
	|a(w,v_{h_k})-a_h(u_{h_k},v_{h_k})| + |(B(w),v)-(B_h(u_{h_k}),v_{h_k})| \rightarrow 0,~h_{k}\rightarrow 0,
\end{align}
To show (\ref{Th32}), we need to make the following estimates for its left-hand side. For the first term, using the Lemma \ref{globalpi}, we have
\begin{align}\label{Th33}
	|a(w,v_{h_k})-a_h(u_{h_k},v_{h_k})|
	&\leq |\sum_Ea^E(w-u_{h_k},v_{h_k})|+|\sum_E[a^E(u_{h_k},v_{h_k})-a^E_h(u_{h_k},v_{h_k})]|\notag\\
    &\leq |\sum_Ea^E(w-u_{h_k},v_{h_k})|+|\sum_Ea^E(u_{h_k}-\Pi_hu_{h_k},v_{h_k})| \notag\\
	&+ |\sum_Ea^E_h(\Pi_hu_{h_k}-u_{h_k},v_{h_k})|\notag\\
	&\leq C||w-u_{h_k}||_{1,\Omega}||v_{h_k}||_{1,\Omega} + Ch_k|logh_k|^{\frac{1}{2}}||u_{h_k}||_{X}||v_{h_k}||_{1,\Omega}.
\end{align}
For second term, similar as the deduction of (\ref{Th17}), we can derive
\begin{align}\label{Th34}
	|(B(w),v)&-(B_h(u_{h_k}),v_{h_k})|
	\leq |(B(w),v-v_{h_k})| + |(B(w),v_{h_k})-(B_h(u_{h_k}),v_{h_k})|\notag\\
	&\leq C||B(w)||_{0,\Omega} ||v-v_{h_k}||_{1,\Omega}+C(h_k||w||_X+||w-u_{h_k}||_{0,\Omega})||v_{h_k}||_{1,\Omega}.
\end{align}
Then, using the fact that $u_{h_k} \rightarrow w$, and $v_{h_k}\rightarrow v$, we know that (\ref{Th32}) holds. Hence
\[
   a(w,v)+(B(w),v)=(f_G,v).
\]
Since $u$ is the unique solution of $(\ref{npbew})$, we get $u=w$. So, it follows that $u_h\rightarrow u$, in $L^2(\Omega)$, thus $||u-u_h||_{0,\Omega}\rightarrow 0$.  $\hfill\Box$
\end{proof}
\\

Then, from Theorems \ref{H1error} - \ref{uuhconvergence} for sufficiently small $h$, we have the following theorem.
\begin{theorem}\label{final}
Let $u$ and $u_h$ be solutions of (\ref{npbew}) and (\ref{nlpbed}), respectively. If $u_h\in W^{1,\infty}(\Omega)$, then for $h$ sufficiently small, we have
\[
	||u-u_h||_{0,\Omega}\leq Ch^2|logh|,~||u-u_h||_{1,\Omega}\leq Ch|logh|^{\frac{1}{2}}.
\]
\end{theorem}

From the deductions of Theorems \ref{H1error} and \ref{L2error},
if there is no interface element, i.e.,
$\mathcal{T}_h^* = \emptyset$, then we have the following corollary.
\begin{corollary}\label{coro}
      Assume $\mathcal{T}_h^* = \emptyset$. Under the same assumpting of Theorem \ref{uuhconvergence}, for $h$ sufficiently small, we have
\[
	||u-u_h||_{0,\Omega} + h||u-u_h||_{1,\Omega} \leq Ch^2.
\]
\end{corollary}
\section{Numerical Results}\setcounter{equation}{0}

In this section, we present the numerical results to illustrate the theoretical results obtained in Section \ref{errorestimates}. The code is written in Fortran 90 and all the computations are carried out on the Ubuntu 18.04.6 LTS with GNU/Linux 4.15.0-193-generic x86\_64. To implement the virtual element method, we refer to \cite{L2017The}.
\begin{example}
    We solve the following regularized Poisson-Boltzmann equation:
\begin{equation}\label{example}
\begin{aligned}
    -\nabla\cdot(\varepsilon\nabla u)+\bar{\kappa}^2\sinh(u+G) &= f + \nabla \cdot((\varepsilon-\varepsilon_{m})\nabla G),\quad\text{in}\;\Omega,\\
    u &= 0,~~\text{on}~\partial \Omega,
\end{aligned}
\end{equation}
where the $f$ is a given function and the domain $\Omega$ is described specifically by
\[
   \Omega:=[0,1]^3,~\Omega_m:=[0,0.5]^3,~\Omega_s = \Omega\setminus \Omega_m.
\]
The parameters defined in (\ref{parameters}) are taken as (cf. \cite{Holst2004The})
\[
   \varepsilon_m = 2,~\varepsilon_s = 80,~\kappa = \frac{1}{20\sqrt{2}}.
\]
For simplicity of calculation, we take on singular point $\boldsymbol{x}_1=(0,0,0)$ in the Green function $G$ defined in (\ref{G}). 
\end{example}
\subsection{Meshes}
In order to observe the performance of virtual element method on different polyhedral meshes, we consider the following three meshes types in the numerical experiment, shown in Figs. \ref{th1}-\ref{th3}.
The first and second diagrams represent tetrahedral and cubic meshes, respectively. The third figure is the random Voronoi meshes where the control points of the Voronoi tessellation are randomly displaced inside the domain.


\begin{figure}[H]
 \centering
 {
  \begin{minipage}{7.5cm}
   \centering
   \includegraphics[scale=0.44]{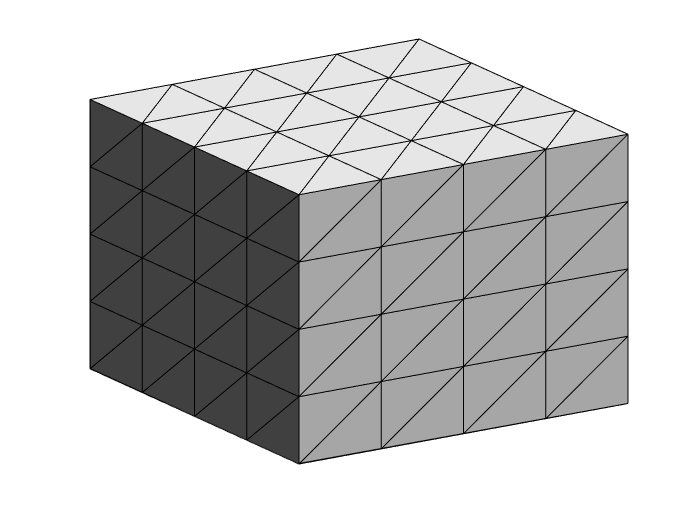}
   \caption{The tetrahedral mesh.}
   \label{th1}
  \end{minipage}
 }
    {
     \begin{minipage}{8cm}
      \centering
      \includegraphics[scale=0.44]{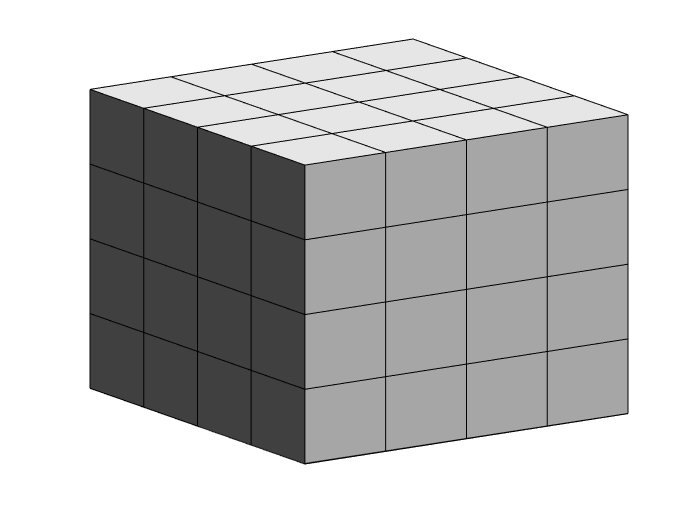}
      \caption{The cubic mesh.}
      \label{th2}
     \end{minipage}
    }
\end{figure}

\begin{figure}[H]
 \centering
 {
  \begin{minipage}{9cm}
   \centering
   \includegraphics[scale=0.48]{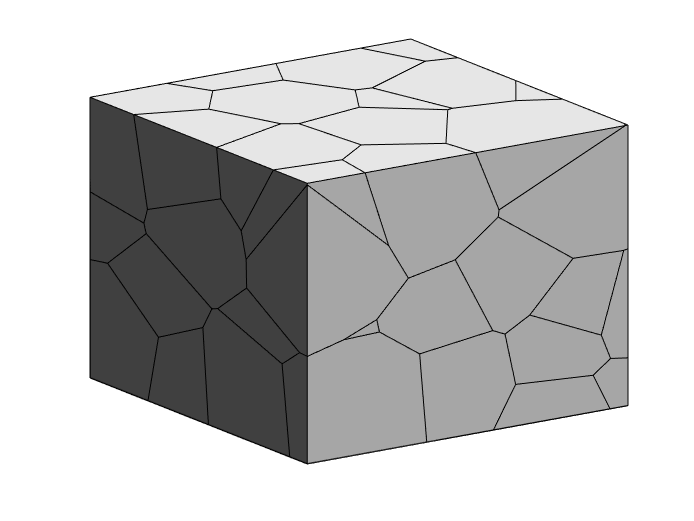}
   \caption{The Voronoi mesh.}
   \label{th3}
  \end{minipage}
 }
\end{figure}

\subsection{Error norms}
Let $u_{ex}$ be the exact solution of the problem (\ref{example}) and $u_h$ be the discrete solution provided by the virtual element method. Since the virtual element solution is not explicitly known inside the elements, in order to estimate how the virtual element solution approaches the exact one, we use the local projectors of degree $k$ on each polyhedron $E$ of the mesh, $``\Pi_k^{\nabla}u_h "$ defined in (\ref{P1.4})-(\ref{P1.6}), where $k=1$. Then we compute the following quantities:
\begin{itemize}
  \item $\mathbf{H^1-seminorm~error}$
  \[
      e_{H^1}:=\sqrt{\sum_{E\in \mathcal{T}_h}|u_{ex}-\Pi_k^{\nabla}u_h|_{1,E}^2},
  \]
  \item $\mathbf{L^2-norm~error}$
  \[
      e_{L^2}:=\sqrt{\sum_{E\in \mathcal{T}_h}||u_{ex}-\Pi_k^{\nabla}u_h||_{0,E}^2}.
  \]
\end{itemize}

For tetrahedral and cubic meshes, we choose the right-hand side $f=0$ in (\ref{example}) such that the problem (\ref{example}) is a standard nonlinear RPBE as (\ref{rrpbe}) and it has no exact solution. So we make the discrete solution ``$u_*$" on the mesh with $Dof = 129^3$($Dof$ is the degree of freedom) as the exact solution of (\ref{example}), i.e.,
\begin{align}\label{uex1}
     u_{ex}:= \Pi_k^{\nabla}u_*.
\end{align}

For testing on the random Voronoi mesh, since the coarse and fine meshes are not nested, it is inaccurate to measure the error using the ``exact" solution $u_*$ on a finer mesh. Therefore, we consider the following exact solution in (\ref{example})
\begin{align}\label{uex2}
      u_{ex} = sin(\pi x)sin(\pi y)sin(\pi z),
\end{align}
and $f:=-\nabla\cdot(\varepsilon\nabla u_{ex})+\bar{\kappa}^2\sinh(u_{ex}+G)$.

In all of the above cases, the mesh-size parameter $h$ is measured in an averaged sense (cf. \cite{L2017High-order})
\begin{align}\label{h}
      h = (\frac{|\Omega|}{N_E})^{\frac{1}{3}},
\end{align}
with $N_E$ denoting the number of polyhedrons in the mesh. The log-log figures will be plotted for the original outputs ($x$-axis denotes the mesh size $h$ and $y$-axis denotes the $L^2$ norm or $H^1$ seminorm of errors).
\subsection{Results}
In this subsection, we present the convergence results about the solution of the virtual element method with the order $k=1$. To confirm the theoretical results, we calculate the rate of convergence by using the following formula
\[
    Order:=\frac{log (e_{i+1}/e_{i})}{log (h_{i+1}/h_i)},~i=1,...,4,
\]
where the $Order$ is $L^2$-Order or $H^1$-Order and $e_{i}$ is a norm error of $e_{L^2}$ or $e_{H^1}$. The $h_i$ is defined in (\ref{h}).

In Tables $\ref{tettable}$ and $\ref{cubetable}$, we give the $L^2$-norm error ($e_{L^2}$), $H^1$-seminorm error ($e_{H^1}$) and the corresponding $L^2$ and $H^1$ error order($L^2$-Order, $H^1$-Order) on tetrahedral and cubic meshes, respectively. 
It is seen from Tables $\ref{tettable}$ and $\ref{cubetable}$ that the convergence orders in $L^2$ norm and $H^1$ norm are near second order and first order, respectively. The numerical results are also shown in Figs. \ref{tetf} and \ref{cubef}, which verify the theoretical results shown in Corollary \ref{coro}.

Fig. \ref{cutvorf} is a two-dimensional section view on the Voronoi mesh. Fig. \ref{cutvorf} indicates that there exists interface elements belong to $\mathcal{T}_h^*$ (defined in (\ref{ge})) on the Voronoi mesh. Fig. \ref{vorf} shows that the error curves of $u_h$ keep the quasi-optimal convergence order with $u_{ex}$ defined in (\ref{uex2}), which verifies the theoretical results shown in Theorem \ref{final}.


\begin{table}[H]
\centering
\caption{The error for the tetrahedral mesh with $u_{ex}$ defined in (\ref{uex1})}
\vskip 0.2cm
  \begin{tabular}{m{2.5cm}<{\centering}m{2.5cm}<{\centering}m{3.0cm}<{\centering}m{2.5cm}<{\centering}m{3.0cm}<{\centering}}
    \toprule
    $Dof$ & $e_{L^2}$ & $L^2$-Order & $e_{H^1}$ & $H^1$-Order \\
    \midrule
    $5^3$  & 5.84e-20 &   -  & 6.52e-19 & - \\
    $9^3$  & 2.68e-20 & 1.22 & 4.39e-19 & 0.57 \\
    $17^3$ & 9.19e-21 & 1.55 & 2.25e-19 & 0.96 \\
    $33^3$ & 2.70e-21 & 1.77 & 1.05e-19 & 1.11 \\
    $65^3$ & 6.29e-22 & 2.10 & 3.91e-20 & 1.42 \\
    \bottomrule
  \end{tabular}\label{tettable}
\end{table}

\begin{table}[H]
\centering
\caption{The error for the cubic mesh with $u_{ex}$ defined in (\ref{uex1})}
\vskip 0.2cm
  \begin{tabular}{m{2.5cm}<{\centering}m{2.5cm}<{\centering}m{3.0cm}<{\centering}m{2.5cm}<{\centering}m{3.0cm}<{\centering}}
    \toprule
    $Dof$ & $e_{L^2}$ & $L^2$-Order & $e_{H^1}$ & $H^1$-Order \\
    \midrule
    $5^3$  & 4.63e-20 &   -  & 3.43e-19 & - \\
    $9^3$  & 2.36e-20 & 0.97 & 2.86e-19 & 0.26 \\
    $17^3$ & 8.17e-21 & 1.53 & 1.56e-19 & 0.88 \\
    $33^3$ & 2.34e-21 & 1.80 & 7.21e-20 & 1.11 \\
    $65^3$ & 5.32e-22 & 2.15 & 2.73e-20 & 1.40 \\
    \bottomrule
  \end{tabular}\label{cubetable}
\end{table}



\begin{figure}[H]
 \centering
 {
  \begin{minipage}{8cm}
   \centering
   \includegraphics[scale=0.5]{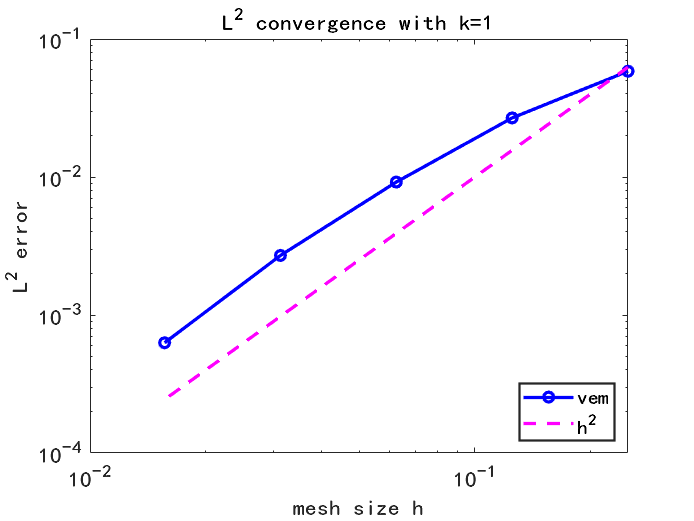}
  \end{minipage}
 }
    {
     \begin{minipage}{8cm}
      \centering
      \includegraphics[scale=0.5]{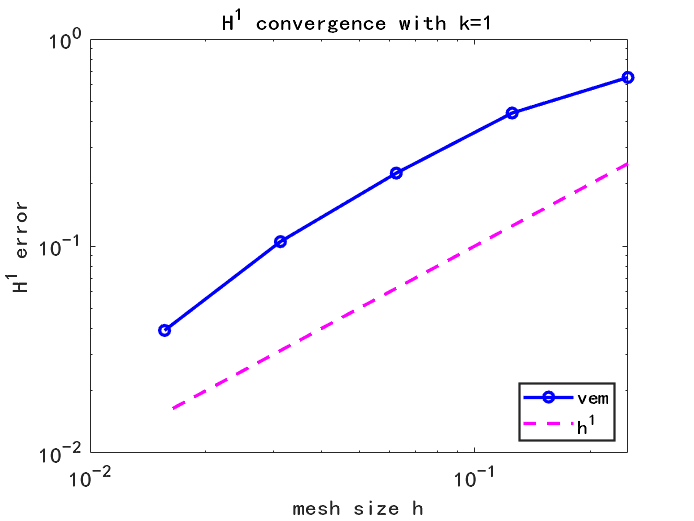}
     \end{minipage}
    }

    \caption{The $L^2$ and $H^1$ norm errors on tetrahedral mesh. The pink dotted line is a quasi-optimal convergence curve with slope $2$ (Left) or $1$ (Right). } \label{tetf}
\end{figure}

\begin{figure}[H]
 \centering
 {
  \begin{minipage}{8cm}
   \centering
   \includegraphics[scale=0.5]{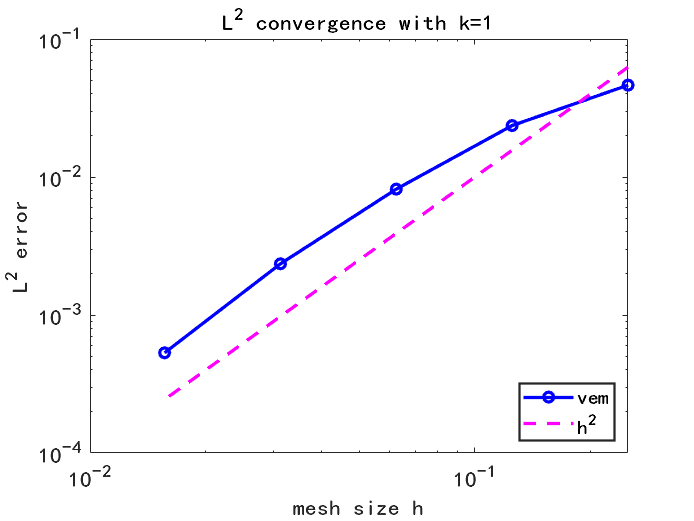}
  \end{minipage}
 }
    {
     \begin{minipage}{8cm}
      \centering
      \includegraphics[scale=0.5]{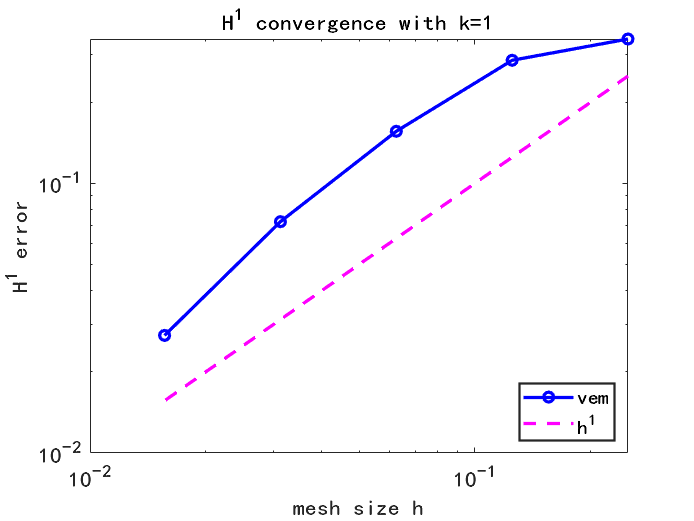}
     \end{minipage}
    }

    \caption{The $L^2$ and $H^1$ norm errors on cubic mesh. The pink dotted line is a quasi-optimal convergence curve with slope $2$ (Left) or $1$ (Right).}\label{cubef}
\end{figure}


\begin{figure}[H]
 \centering
 {
  \begin{minipage}{10cm}
   \centering
   \includegraphics[scale=0.37]{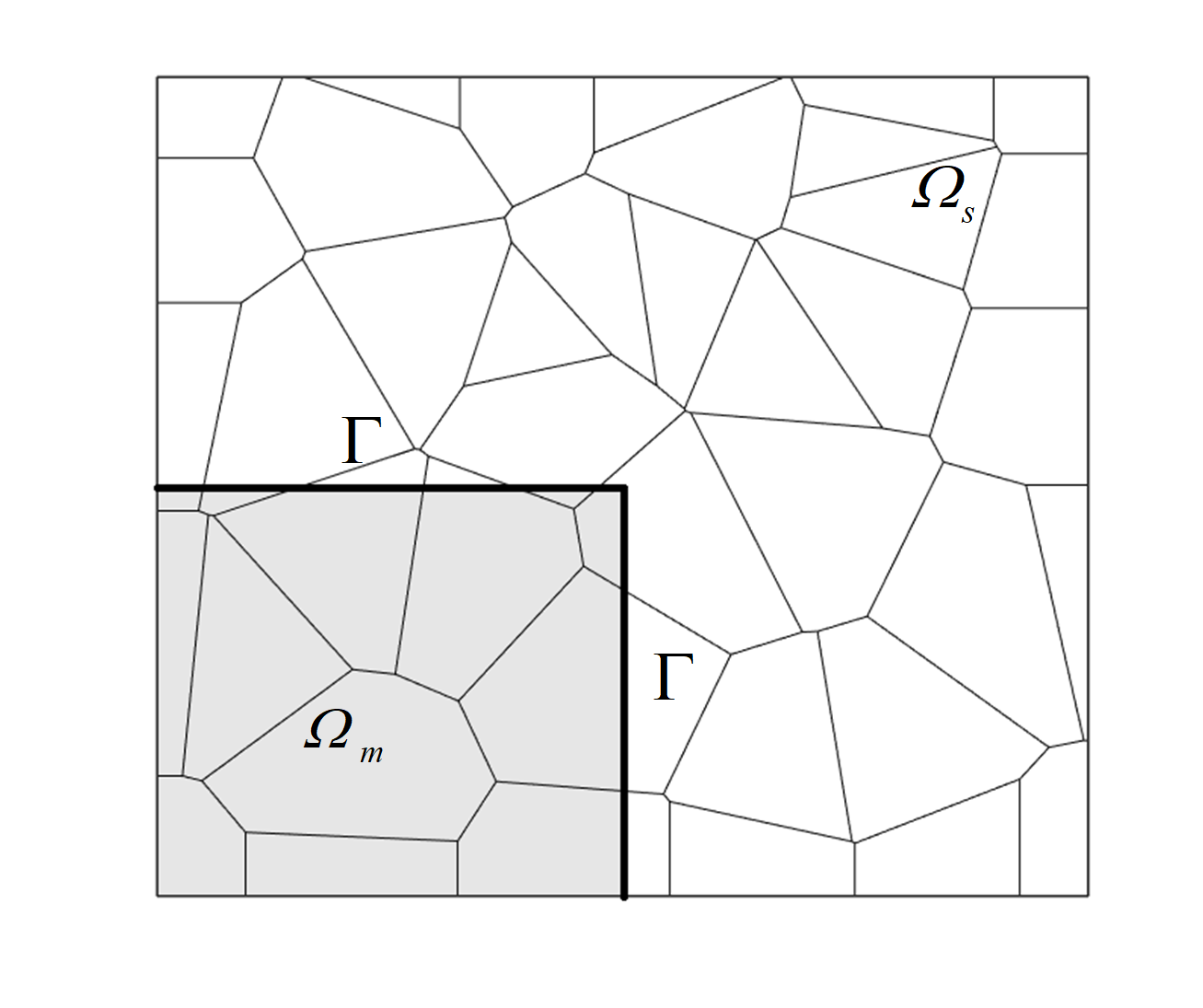}
  \end{minipage}
 }

    \caption{A 2D cut plant through the center of the simulation box along the z axis. The grey area is the 2D view of area $\Omega_m$ on the cross section and the rest is $\Omega_s$. The bold line is the interface $\Gamma$. }\label{cutvorf}
\end{figure}

\begin{figure}[H]
 \centering
 {
  \begin{minipage}{8cm}
   \centering
   \includegraphics[scale=0.5]{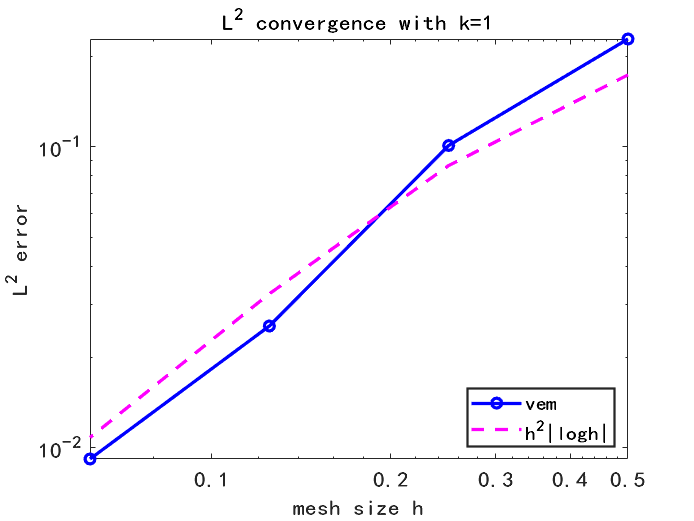}
  \end{minipage}
 }
    {
     \begin{minipage}{8cm}
      \centering
      \includegraphics[scale=0.5]{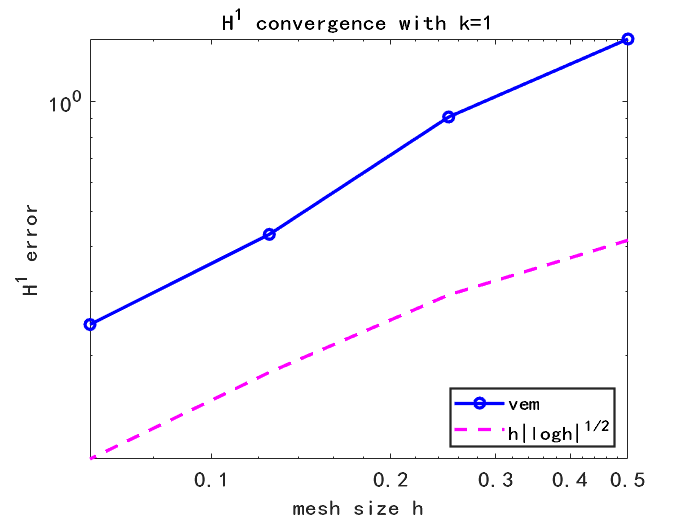}
     \end{minipage}
    }

    \caption{The $L^2$ and $H^1$ norm errors on Voronoi mesh. The pink dotted line is a quasi-optimal convergence curve shown in Theorem \ref{final}.}\label{vorf}
\end{figure}

\section{Conclusion} \label{sec-conclusion}

\noindent

In this paper, we propose a virtual element method to solve the PBE in three dimensions and present the nearly optimal error estimates in $H^1$-norm and $L^2$-norm on the general polyhedral mesh, respectively. The numerical example on different polyhedral meshes confirms the validity of the theoretical results and
shows the efficiency of the virtual element method. This method can be further applied to more complex PBE models, such as spherical interface model and biological protein molecular model, which are our future work.


\section*{Acknowledgments}
The authors would like to thank Jianhua Chen and Yang Liu for their valuable discussions on numercial experiments. Y. Yang was supported by the China NSF(NSFC12161026), Guangxi Natural Science Foundation(2020GXNSFAA159098). S. Shu was supported by the China NSF (NSFC 11971414).


\bibliographystyle{abbrv}     

\bibliography{ref}   


%
%

\end{document}